\documentclass[10pt]{article}

\usepackage[a4paper,margin=2.9cm]{geometry}
\usepackage[english]{babel}
\usepackage{parskip}
\usepackage{amssymb}
\usepackage{amsmath, amsthm, amssymb}
\usepackage{enumerate}
\usepackage{graphicx}
\usepackage{bbm} 
\usepackage{mathrsfs} 
\usepackage{subcaption}
\usepackage{verbatim}
\usepackage{hyperref}
\usepackage[percent]{overpic}

\makeatletter
\renewcommand\section{\@startsection{section}{1}{\z@}%
                                  {-3.5ex \@plus -1ex \@minus -.2ex}%
                                  {2.3ex \@plus.2ex}%
                                  {\normalfont\Large\bfseries}}
\makeatother

\newtheoremstyle{examplestyle}
  {4mm}
  {4mm}
  {\slshape}
  {0pt}
  {\bfseries}
  {\newline}
  {0mm}
  {}

\theoremstyle{examplestyle}
  \newtheorem{Theorem}{Theorem}[section]              

  \newtheorem{Lemma}[Theorem]{Lemma}
  
 \newtheorem*{Defs*}{Definitions}

  \newtheorem{Corollary}[Theorem]{Corollary}

\newtheorem{Conjecture}[Theorem]{Conjecture}
\newtheorem{Claim}[Theorem]{Claim}

\title{Packing graphs of bounded codegree}

\author{Wouter Cames van Batenburg\thanks{Department of Mathematics, Radboud University Nijmegen, Postbus 9010, 6500 GL Nijmegen, The Netherlands. \href{mailto:w.camesvanbatenburg@math.ru.nl}{\nolinkurl{w.camesvanbatenburg@math.ru.nl}}, \href{mailto:ross.kang@gmail.com}{\nolinkurl{ross.kang@gmail.com}}.} \and Ross J. Kang\footnotemark[1]}

\begin{document}

\maketitle

\begin{abstract}
Two graphs $G_1$ and $G_2$ on $n$ vertices are said to \textit{pack} if there exist injective mappings of their vertex sets into $[n]$ such that the images of their edge sets are disjoint. A longstanding conjecture due to Bollob\'as and Eldridge and, independently, Catlin, asserts that, if $(\Delta_1(G)+1) (\Delta_2(G)+1) \le n+1$, then $G_1$ and $G_2$ pack. We consider the validity of this assertion under the additional assumption that $G_1$ or $G_2$ has bounded codegree. In particular, we prove for all $t \ge 2$ that, if $G_1$ contains no copy of the complete bipartite graph $K_{2,t}$ and $\Delta_1 > 17  t \cdot \Delta_2$, then $(\Delta_1(G)+1) (\Delta_2(G)+1) \le n+1$ implies that $G_1$ and $G_2$ pack.
We also provide a mild improvement if moreover $G_2$ contains no copy of the complete tripartite graph $K_{1,1,s}$, $s\ge 1$.
\end{abstract}

\section{Introduction}

Let $G_1$ and $G_2$ be graphs on $n$ vertices. (All graphs are assumed to have neither loops nor multiple edges.)
We say that $G_1$ and $G_2$ {\em pack} if there exist injective mappings of their vertex sets into $[n] = \{1,\dots,n\}$ so that their edge sets have disjoint images. Equivalently, $G_1$ and $G_2$ pack
if $G_1$ is a subgraph of the complement of $G_2$.
The {\em maximum codegree $\Delta^\wedge(G)$} of a graph $G$ is the maximum over all vertex pairs of their common degree, i.e.~$\Delta^\wedge(G) < t$ if and only if $G$ contains no copy of the complete bipartite graph $K_{2,t}$.
The {\em maximum adjacent codegree $\Delta^\vartriangle(G)$} of $G$ is the maximum over all pairs of {\em adjacent} vertices of their common degree, i.e.~$\Delta^\vartriangle(G) < s$ if and only if $G$ contains no copy of the complete tripartite graph $K_{1,1,s}$.
Clearly, $\Delta^\vartriangle(G) \le \Delta^\wedge(G)$ always.
We let $\Delta_1$ and $\Delta_2$ denote the maximum degrees of $G_1$ and $G_2$, respectively, and $\Delta^\wedge_1$ and $\Delta^\vartriangle_2$ the corresponding maximum (adjacent) codegrees.
We provide sufficient conditions for $G_1$ and $G_2$ to pack  in terms of $\Delta_1$, $\Delta_2$, $\Delta^\wedge_1$, $\Delta^\vartriangle_2$. 

For integers $t \ge 2$ and $\Delta_2 \ge 1$, we define
\begin{align*}
\alpha^*(t,\Delta_2):= \frac{1}{2}(2+ \gamma + \sqrt{4\gamma + \gamma^2}),\ \text{ where } \ \gamma=\frac{\Delta_2}{\Delta_2+1}\cdot\frac{t-1}{t}.
\end{align*}
Note $\alpha^*$ is the larger solution to the equation $(\alpha-1)^2 - \gamma \alpha=0$ and $\frac{1}{8}(9+\sqrt{17}) \le \alpha^* \le \frac{1}{2}(3+\sqrt{5})$.

\begin{Theorem}\label{thm:blue}
Let $G_1$ and $G_2$ be graphs on $n$ vertices with respective maximum degrees $\Delta_1$ and $\Delta_2$. Let $\Delta^\wedge_1$ be the maximum codegree of $G_1$.
Let $t\ge 2$ be an integer and let $\alpha > \alpha^*=\alpha^*(t,\Delta_2)$ and $0 < \epsilon <1/2$ be reals.
Then $G_1$ and $G_2$ pack
if $\Delta^\wedge_1<t$ and $n$ is larger than each of the following quantities:
\begin{align}
\left(t+\frac{\alpha (\alpha-1)}{(\alpha-1)^2-\alpha} \right) \cdot \Delta_2 + \Delta_1 \Delta_2, &\label{eqn:blue1}\\
(2\alpha t+2)\cdot \Delta_2 + ((2\alpha+1)t-1) \cdot \Delta_2^2 + (1-\epsilon) \cdot \Delta_1\Delta_2 , &\label{eqn:blue2}\\
 1 + \left( 2+ \frac{\epsilon}{1-2\epsilon}\right) \cdot \Delta_2 + \Delta_1 \Delta_2, & \ \text{ and}\label{eqn:blue3}\\
 \left( t + \frac{3 - \epsilon}{2}  \right) \cdot \Delta_2  + \frac{3-\epsilon}{2} (t-1) \cdot \Delta_2^2 + \frac{1+ \epsilon}{2} \cdot \Delta_1\Delta_2 \label{eqn:blue4}.
\end{align}
\end{Theorem}

\begin{Theorem}\label{thm:bluered}
Let $G_1$ and $G_2$ be graphs on $n$ vertices with respective maximum degrees $\Delta_1$ and $\Delta_2$. Let $\Delta^\wedge_1$ be the maximum codegree of $G_1$ and $\Delta^\vartriangle_2$ the maximum adjacent codegree of $G_2$.
Let $s\ge1$ and $t\ge 2$ be integers and let $\alpha > \alpha^*=\alpha^*(t,\Delta_2)$ be real.
Then $G_1$ and $G_2$ pack
if $\Delta^\wedge_1<t$, $\Delta^\vartriangle_2<s$, and $n$ is larger than both of the following quantities:
\begin{align}
\left(t+\frac{\alpha (\alpha-1)}{(\alpha-1)^2-\alpha} \right) \cdot \Delta_2 + \Delta_1 \Delta_2 & \ \text{ and} \label{eqn:bluered1}\\
 (2 + 2\alpha t) \cdot \Delta_2 + (s-1) \cdot \Delta_1 + \left( (2\alpha+1)t-1\right) \cdot \Delta_2^2 &   \label{eqn:bluered2}.
\end{align}
\end{Theorem}

For better context, we compare Theorems~\ref{thm:blue} and~\ref{thm:bluered} to a line of work on graph packing that was initiated in the 1970s~\cite{BoEl78,Cat74,Cat76,SaSp78}.
The following is a central problem in the area.

\begin{Conjecture}[Bollob\'as and Eldridge~\cite{BoEl78} and Catlin~\cite{Cat76}]
Let $G_1$ and $G_2$ be graphs on $n$ vertices with respective maximum degrees $\Delta_1$ and $\Delta_2$. Then $G_1$ and $G_2$ pack if $(\Delta_1 + 1) (\Delta_2 +1) \le n+1$. 
\end{Conjecture}
If true, the statement would be sharp and would significantly generalise a celebrated result of Hajnal and Szemer\'edi~\cite{HaSz70} on equitable colourings.
Sauer and Spencer~\cite{SaSp78} showed that $2 \Delta_1 \Delta_2 < n$ is a sufficient condition for $G_1$ and $G_2$ to pack, which is seen to be sharp when one of the graphs is a perfect matching.
Thus far the Bollob\'as--Eldridge--Catlin (BEC) conjecture has been confirmed in the following special cases: $\Delta_1 = 2$~\cite{AiBr93}; $\Delta_1 = 3$ and $n$ sufficiently large~\cite{CSS03}; $G_1$ bipartite and $n$ sufficiently large~\cite{Csa07}; and $G_1$ $d$-degenerate, $\Delta_1 \ge 40 d$ and $\Delta_2 \ge 215$~\cite{BKN08}.
Moreover, an approximate BEC condition, $(\Delta_1 + 1) (\Delta_2 +1) \le 3n/5+1$, is sufficient for $G_1$ and $G_2$ to pack, provided that $\Delta_1,\Delta_2\ge 300$~\cite{KKY08}.
Theorem \ref{thm:blue} implies the following.

\begin{Corollary}\label{cor:blue0}
Let $G_1, G_2, \Delta_1$,$\Delta_2$ and $\Delta^\wedge_1$ be as before. Let $t\ge 2$ be an integer. Then $G_1$ and $G_2$ pack if $\Delta_1 \Delta_2 + \Delta_1 \leq n+1$ and $\Delta^\wedge_1 < t$ and $\Delta_1 > 17 t \cdot \Delta_2$.
\begin{proof}
Choose $\epsilon = (2t-2)/(4t-3)$ and $\alpha=3$ in Theorem~\ref{thm:blue}. Using that $\Delta_1 > 17 t \Delta_2> \frac{(4t-3)(7t-1)}{2t-2} \cdot \Delta_2$, it follows that 
$\max(\eqref{eqn:blue1},\eqref{eqn:blue2},\eqref{eqn:blue3},\eqref{eqn:blue4}) \le (\Delta_1+1)(\Delta_2+1)-1  \le  n$. So $G_1$ and $G_2$ pack.
\end{proof}
\end{Corollary}

The following results concerning the BEC-conjecture follow immediately.
\begin{Corollary}\label{cor:blue}
Given an integer $t\ge 2$, the BEC conjecture holds under the additional condition that the maximum codegree $\Delta^\wedge_1$ of $G_1$ is less than $t$ and $\Delta_1 > 17 t \cdot \Delta_2$.
\end{Corollary}

We were unable to avoid the linear dependence on $\Delta_2$ in the lower bound condition on $\Delta_1$.
Although we have not seriously attempted to optimise the factor $17 t$ above, Theorem~\ref{thm:bluered} improves on this factor  under the additional assumption that $\Delta^\vartriangle_2$ is bounded, as exemplified by the following corollary.

\begin{Corollary}\label{cor:bluered}
Given an integer $t\ge 2$, the BEC conjecture holds under the additional condition that the maximum codegree $\Delta^\wedge_1$ of $G_1$ is less than $t$, $G_2$ is triangle-free, and $\Delta_1 > (4+\sqrt{5}) t \cdot \Delta_2$.
\begin{proof}
Choose $\alpha=\frac{1}{4t}(6t+1 + \sqrt{20 t^2 + 4t+1})$ and $s=1$ in Theorem~\ref{thm:bluered}. Using that $t+ \frac{\alpha (\alpha-1)}{(\alpha-1)^2-\alpha} -1 = (2\alpha+1) t -1$ and that $\Delta_1 > (4+\sqrt{5}) t \cdot \Delta_2 > ((2 \alpha +1) t-1) \cdot \Delta_2$, it follows that $\max(\eqref{eqn:bluered1},\eqref{eqn:bluered2})  \le (\Delta_1+1)(\Delta_2+1)-1  \le  n$. So $G_1$ and $G_2$ pack.
\end{proof}
\end{Corollary}

\subsection*{Structure of the paper}

In the next section, we provide some notation and preliminary observations. In Section~\ref{sec:critical}, we discuss the common features of a hypothetical critical counterexample to one of our theorems. In Section~\ref{sec:proofs}, we prove Theorems~\ref{thm:blue} and~\ref{thm:bluered}. We conclude the paper with some remarks about the results, proofs and further possibilities.

\section{Notation and preliminaries}\label{sec:preliminaries}

Here we introduce some terminology which we use throughout.
We often call $G_1$ the {\em blue} graph and $G_2$ the {\em red} graph.
We treat the injective vertex mappings as labellings of the vertices from $1$ to $n$.
However, rather than saying, ``the vertex in $G_1$ (or $G_2$) corresponding to the label $i$'', we often only say, ``vertex $i$'', since this should never cause any confusion.
Our proofs rely on accurately specifying the neighbourhood structure as viewed from a particular vertex. Let $i\in [n]$.
The {\em blue neighbourhood $N_1(i)$ of $i$} is the set $\{j \,\mid\, ij \in E(G_1)\}$ and the {\em blue degree $\deg_1(i)$ of $i$} is $|N_1(i)|$.
The {\em red neighbourhood $N_2(i)$} and {\em red degree $N_2(i)$} are defined analogously.
For $j\in[n]$, a {\em red--blue-link (or $2$--$1$-link) from $i$ to $j$} is a vertex $i'$ such that $ii' \in E(G_2)$ and $i'j \in E(G_1)$.
The {\em red--blue-neighbourhood $N_1(N_2(i))$ of $i$} is the set $\{j \,\mid\, \exists\text{ red--blue-link from $i$ to $j$}\}$.
A {\em blue--red-link (or $1$--$2$-link)} and the {\em blue--red-neighbourhood} $N_2(N_1(i))$ are defined analogously.

In search of a certificate that $G_1$ and $G_2$ pack, without loss of generality, we keep the vertex labelling of the blue graph $G_1$ fixed, and permute only the labels in the red graph $G_2$. This can be thought of as ``moving'' the red graph above a fixed ground set $[n]$.
In particular, we seek to avoid the situation that there are $i,j \in [n]$ for which $ij$ is an edge in both $G_1$ and $G_2$ --- in this situation, we call $ij$ a {\em purple} edge induced by the labellings of $G_1$ and $G_2$.
So $G_1$ and $G_2$ pack if and only if they admit a pair of vertex labellings that induces no purple edge.
In our search, we make small cyclic sub-permutations of the labels (of $G_2$), which are referred to as follows.
For $i_0,\dots,i_{\ell-1} \in [n]$, a {\em $(i_0,\dots,i_{\ell-1})$-swap} is a relabelling of $G_2$ so that for each $k \in \{0,\dots,\ell-1\}$ the vertex labelled $i_k$ is re-assigned the label $i_{k+1\bmod \ell}$. In fact, we shall only require swaps having $\ell \in \{1,2\}$.
The following observation describes when a swap could be helpful in the search for a packing certificate.
This is identical to Lemma~1 in~\cite{KKY08}.

\begin{Lemma}\label{lem:swap}
Let $u_0,\dots,u_{\ell-1} \in [n]$. For every $k, k' \in \{0,\dots,\ell-1\}$, suppose that there is no red--blue-link from $u_k$ to $u_{k+1\bmod \ell}$ and that, if $u_k u_{k'} \in E(G_2)$, then $u_{(k+1 \bmod \ell)} u_{(k'+1 \bmod \ell)} \notin E(G_1)$. Then there is no purple edge incident to any of $u_0,\dots,u_{\ell-1}$ after a $(u_0,\dots,u_{\ell-1})$-swap. \qed
\end{Lemma}

We will use a classic extremal set theoretic result to upper bound the size of certain vertex subsets.

\begin{Lemma}[Corr\'adi~\cite{Cor69}]
Let $A_1,\ldots, A_N$ be $k$-element sets and $X$ be their union. If $|A_i \cap A_j| \le t-1$ for all $ i\neq j$, then 
$|X| \ge k^2 N/(k+ (N-1)(t-1))$. \qed
\end{Lemma}
In particular, this implies the following.
\begin{Corollary}\label{CorCor}
Let $A_1,\ldots, A_N$ be size $\ge k$ subsets of a set $X$. If $k^2 > (t-1) \cdot |X|$ and $|A_i \cap A_j| \le t-1$ for all $i \neq j$, then
\begin{align*}
N \le  |X| \cdot \frac{k- (t-1)}{k^2-(t-1) \cdot |X|}.
\end{align*}
\begin{proof}
Consider arbitrary subsets $A_1^*\subset A_1,\ldots, A_N^* \subset A_N$ of size $k$. An application of Corr\'adi's lemma to $A_1^*,\ldots, A_N^*$ yields that
$|X| \ge k^2\cdot N/(k +(N-1)(t-1))$, which is easily seen to be equivalent to $(k^2-(t-1) \cdot |X|) \cdot N \le (k-t) \cdot |X|.$
The corollary follows after dividing both sides of the inequality by $k^2 - (t-1) \cdot |X|$. Note that this division does not cause a sign change because of the assumption that $k^2 > (t-1) \cdot |X|$.
\end{proof}
\end{Corollary}

\section{Hypothetical critical counterexamples}\label{sec:critical}

The overall proof structure we use for both theorems is the same, and in this section we describe common features and some further notation.
Suppose the theorem (one of Theorem~\ref{thm:blue} or~\ref{thm:bluered}) is false. Then there must exist a counterexample, that is, a pair $(G_1,G_2)$ of non-packable graphs on $n$ vertices that satisfy the conditions of the theorem.

Moreover, we may assume that $(G_1,G_2)$ is a {\em critical} pair in the sense that $G_2$ is edge-minimal among all counterexamples. In other words, $G_1$ and $G_2-e$ pack for {\em any} $e\in E(G_2)$. There is no loss of generality, since the removal of an edge from $G_2$ increases neither $\Delta_2$ nor $\Delta^\vartriangle_2$ and obviously affects none of $\Delta_1$, $\Delta^\wedge_1$ and $n$, thus maintaining the required conditions.

Now choose \textit{any} edge $e = uv\in E(G_2)$.
Criticality implies that there is a pair of labellings of $G_1$ and $G_2$ such that $e$ is the {\em unique} purple edge, for otherwise $G_1$ and $G_2-e$ do not pack. Let us fix such a pair of labellings so that we can further describe the neighbourhood structure as viewed from $u$ (or $v$). Estimation of the sizes of subsets in this neighbourhood structure is our main method for deriving upper bounds on $n$ that in turn yield the desired contradiction from which the theorem follows.

We need the definition of the following vertex subsets (which are analogously defined for $v$ also):
\begin{align*}
A(u)
&:= N_2(N_1(u))\setminus (N_1(u)\cup N_2(u)\cup N_1(N_2(u)) ), \\
B(u)
&:= N_1(N_2(u))\setminus (N_1(u)\cup N_2(u)\cup N_2(N_1(u)) ), \\
A^*(u) 
&:=N_2(N_1(u)) \setminus ( N_2(u)\cup N_1(N_2(u)) ),
 \text{ and}\\
N_1^*(u)
&:= N_1(u) \cap ( N_1(N_2(u)) \setminus (N_2(u) \cup N_2(N_1(u))) ).
\end{align*}
One justification for specifying the above subsets is that the following two claims (which are essentially Claims~1 and~2 in~\cite{KKY08}) hold.

\begin{Claim}\label{clm:links1}
For all $w\in [n]\setminus\{v\}$, there is a red--blue-link or a blue--red-link from $u$ to $w$.
\end{Claim}

\begin{proof}
If not, then by Lemma~\ref{lem:swap}, a $(u,w)$-swap yields a new labelling such that $uv$ is not purple anymore and no new purple edges are created. Thus $G_1$ and $G_2$ pack, a contradiction.  See Figure~\ref{fig:links1}.
\end{proof}

\begin{figure}
 \begin{center}
   \begin{overpic}[width=0.54\textwidth]{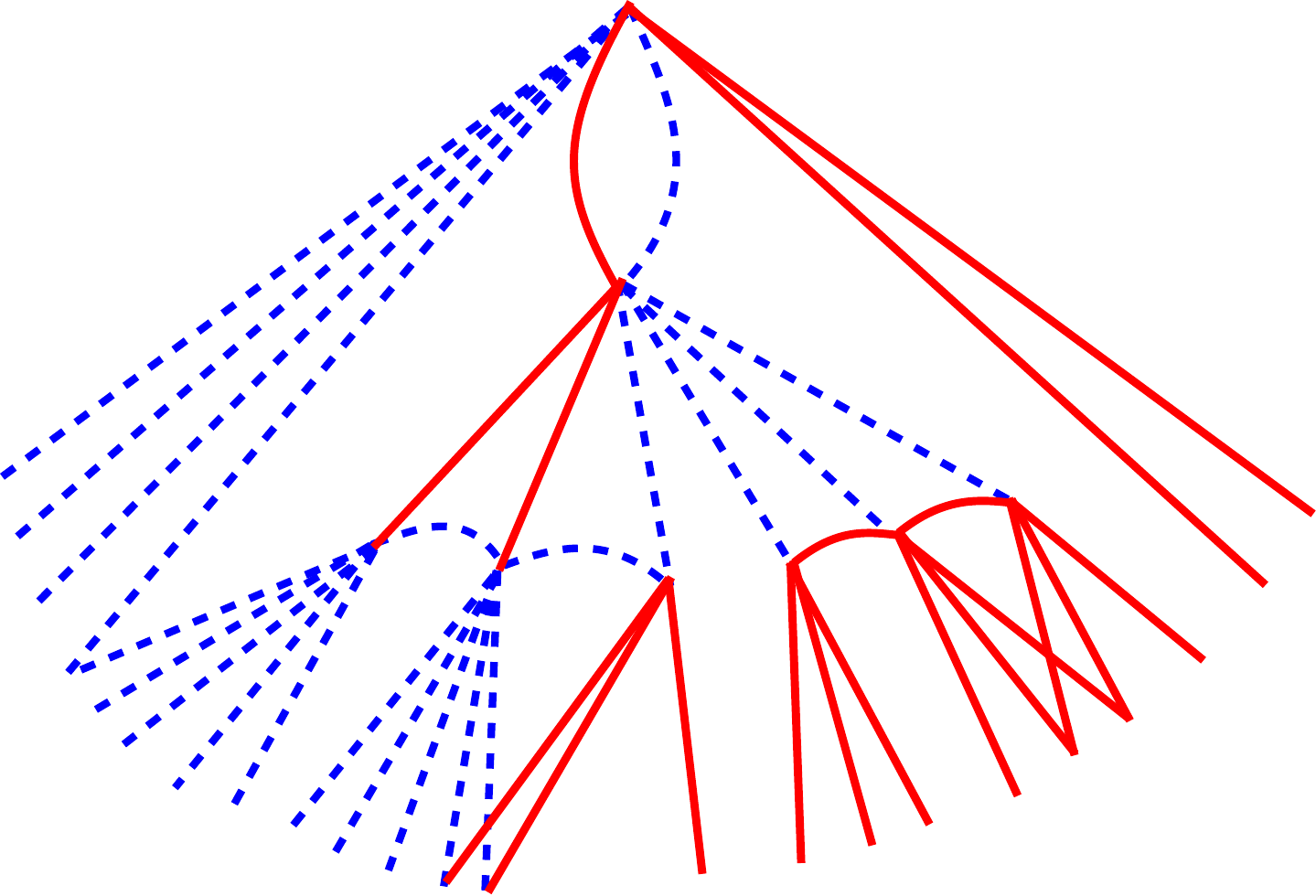}
    \put (45,61){ $v$}
    \put (45,49){ $u$}
 \end{overpic}
 \caption{All vertices (except possibly $v$) are reachable by a link from $u$ (Claim~\ref{clm:links1}).\label{fig:links1}}
\end{center}
\end{figure}

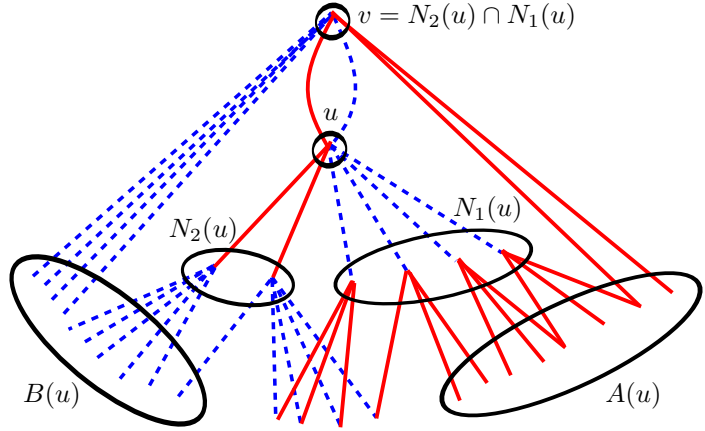
\begin{figure}
 \begin{center}
   \begin{overpic}[width=0.6\textwidth]{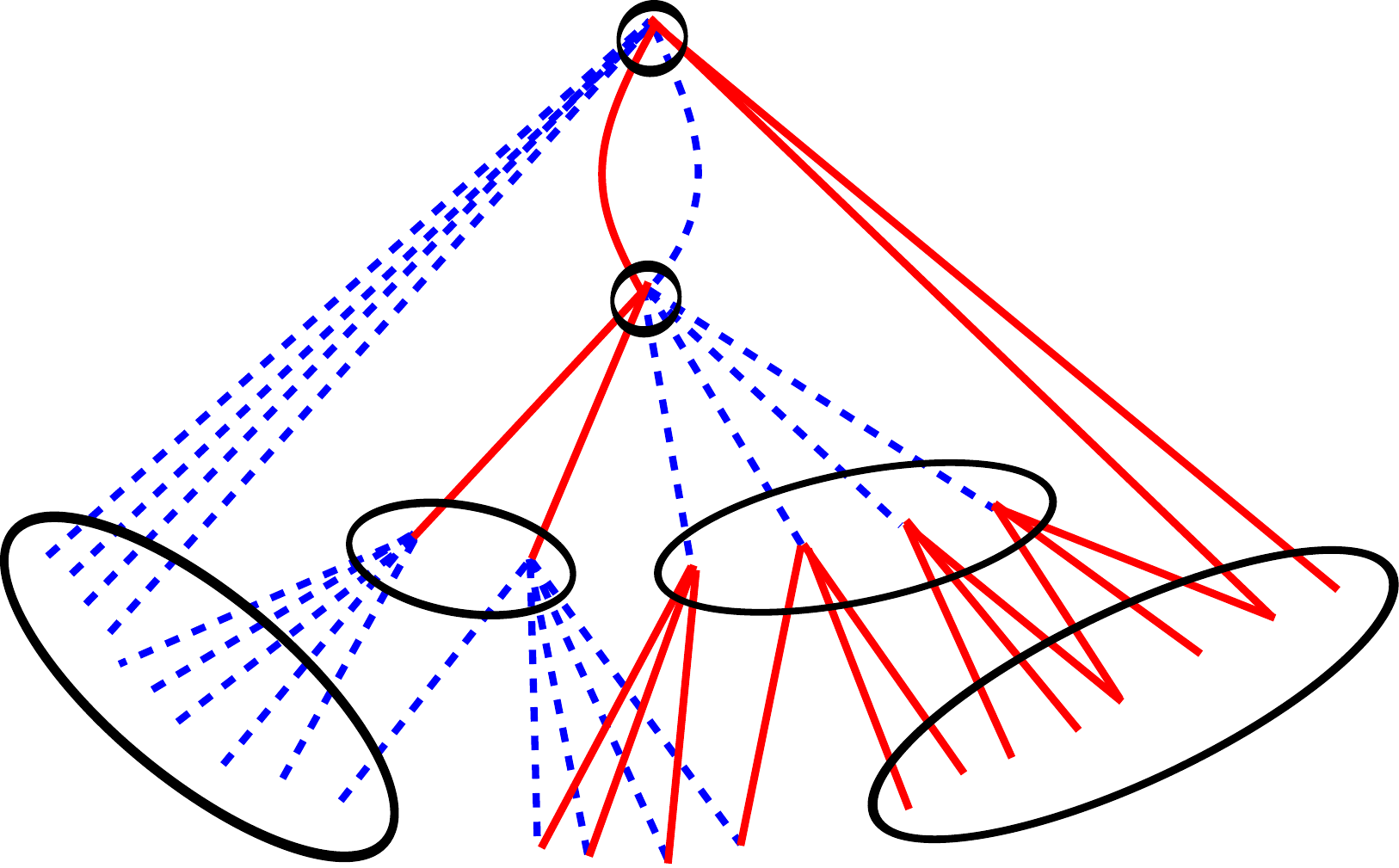}
    \put (49,59){ $v=N_2(u)\cap N_1(u)$}
    \put (44,45){ $u$}
    \put (23,28){$N_2(u)$}
    \put (64,31){$N_1(u)$}
    \put (2,4) {$B(u)$}
    \put (86,4) {$A(u)$}
 \end{overpic}
 \caption{The neighbourhood structure of a hypothetical critical counterexample, as seen from $u$.\label{fig:neighbourhoodstructure}}
\end{center}
\end{figure}

\begin{Claim}\label{clm:links2}
For all $a\in A^*(u)$ and $b\in B(u)$, there is a red--blue-link from $a$ to $b$.
\end{Claim}

\begin{proof}
Since $B(u)\cap N_1(u) = B(u) \cap N_2(u)= \emptyset$ and $A^*(u) \cap N_2(u) = \emptyset$, we have that $bu \notin E(G_1) \cup E(G_2)$ and $ua \notin E(G_2)$. Furthermore, since $A^*(u)\cap N_1(N_2(u))= B(u)\cap N_2(N_1(u))=\emptyset$, there is no red--blue-link from $u$ to $a$ or from $b$ to $u$. Now suppose that there is also no red--blue-link from $a$ to $b$. Then it follows from Lemma~\ref{lem:swap} that after a $(u,a,b)$-swap there is no purple edge incident to any of $u,a,b$, which implies that there is no purple edge at all.
So we have obtained a packing of $G_1$ and $G_2$, a contradiction.
\end{proof}

In the next claim, we list three upper bounds on the total number $n$ of vertices in terms of the sizes of the vertex subsets defined above.  In the proofs of Theorems~\ref{thm:blue} and~\ref{thm:bluered}, we consider several cases for which we prove at least one of these upper bounds to be small enough for a contradiction with the assumed lower bounds on $n$.

\begin{Claim}\label{clm:UpperBoundsN}
The total number $n$ of vertices is at most each of the following quantities:
\begin{enumerate}
\item\label{clm:UpperBoundsN1} $|N_2(u)| + |A^*(u)| + |N_1(N_2(u))|$,
\item\label{clm:UpperBoundsN2} $|N_1^*(u)| + |N_2(u)| + |B(u)| + |N_2(N_1(u))|$,
\item\label{clm:UpperBoundsN3} $|A^*(v)| + |A^*(u)| +  |\left(N_2(u) \cup N_1(N_2(u)) \right) \cap \left( N_2(v) \cup N_1(N_2(v)) \right) |$.
\end{enumerate}
\end{Claim}
\begin{proof}
In all cases, $[n]$ equals the union of the neighbourhood sets that occur in the upper bound.
\begin{enumerate}
\item[\ref{clm:UpperBoundsN1}] The union of $N_2(u)$, $A^*(u)$ and $N_1(N_2(u))$ covers $\left\{v\right\} \cup N_2(N_1(u)) \cup N_1(N_2(u))$, which by Claim~\ref{clm:links1} equals $[n]$.
\item[\ref{clm:UpperBoundsN2}] The union of $N_1^*(u)$, $N_2(u)$, $B(u)$ and $N_2(N_1(u))$ covers $\left\{v\right\} \cup N_2(N_1(u)) \cup N_1(N_2(u))$, which equals $[n]$.
\item[\ref{clm:UpperBoundsN3}] By the proof of (i), $[n]$ is the union of $A^*(u)$ and $ N_2(u) \cup N_1(N_2(u)) $ as well as the union of $A^*(v)$ and $ N_2(v) \cup N_1(N_2(v)) $. It follows that $[n]$ also is the union of $A^*(u)$, $A^*(v)$ and $\left(N_2(u) \cup N_1(N_2(u)) \right) \cap \left( N_2(v) \cup N_1(N_2(v)) \right)$.\qedhere
\end{enumerate}
\end{proof}

The reason for working with $N_1^*(u)$ and $A^*(u)$ rather than the simpler sets $N_1(u)$ and $A(u)$ is the following. Under the requirement that the codegree $\Delta_1^\wedge$ of $G_1$ is less than $t$, we can upper bound $|N_1^*(u)|$ entirely in terms of $\Delta_2$. This is sharper than the trivial bound $|N_1(u)| \le \Delta_1$ because we work under conditions with $\Delta_1$ rather larger than $\Delta_2$. Similarly, since $N_1^*(u) \subset N_1(u)$, we need to compensate for the loss of covered vertices by working with the slightly enlarged set $A^*(u)$, rather than $A(u)$.
The following claims use the condition $\Delta_1^\wedge < t$  (which is assumed by both theorems).

\begin{Claim}\label{clm:truncatedblueneighbourhood}
$|N_1^*(u)|\le (t-1) \cdot \Delta_2$.
\end{Claim}

\begin{proof}
Suppose $|N_1(u) \cap N_1(N_2(u))| \ge  (t-1) \cdot \Delta_2 + 1$, then there is at least one $x\in N_2(u)$ such that $|N_1(u) \cap N_1(x)| \ge \frac{1}{|N_2(u)|} \cdot \left( (t-1) \cdot \Delta_2 + 1 \right) > t-1$, which contradicts $\Delta_1^\wedge < t$.
\end{proof}

The following claim (in combination with Corr\'adi's lemma) is useful for an upper bound on $|B(u)|$ that is only linear in $\Delta_2$,  provided that $|A^*(u)|$ is at least quadratic in $\Delta_2$. See Case~\ref{case:blue1} in the proof of Theorem~\ref{thm:blue}.

\begin{Claim}\label{clm:kclaim}
For any $b\in B(u)$, $|N_1(b)\cap A^*(u)|\ge |A^*(u)|/\Delta_2 -t (\Delta_2+1)$.
\end{Claim}
\begin{proof}
For all $b\in N_1(N_2(u))$ it holds that $|N_1(b)\cap N_1(N_2(u))| \le (t-1) \cdot |N_2(u)| \le (t-1) \cdot \Delta_2$. Indeed, otherwise there would exist a blue copy of $K_{2,t}$ in the graph induced by $N_1(N_2(u))\cup N_2(u)$. Similarly, $|N_1(b)\cap N_1(u)| \le t$ and $|N_1(b)\cap N_2(u)| \le \Delta_2$. So for every $b\in N_1(N_2(u))$, at most $t\cdot (\Delta_2 + 1)$ blue neighbours of $b$ are in $[n] \setminus A(u)$. So in particular, for every $b\in B(u)$, at most $t\cdot (\Delta_2 + 1)$ blue neighbours of $b$ are in $[n] \setminus A^*(u)$.

Using Claim~\ref{clm:links2} and the fact that each blue neighbour of a fixed $b\in B(u)$ has at most $\Delta_2$ red neighbours in $A^*(u)$, we see that every $b\in B(u)$ has at least $\lceil |A^*(u)|/\Delta_2 \rceil$ blue neighbours, and thus at least $|A^*(u)|/\Delta_2  -t (\Delta_2+1)$ blue neighbours in $A^*(u)$.
\end{proof}

\section{Proofs}\label{sec:proofs}

\subsection{Proof of Theorem~\ref{thm:blue}}

Suppose the theorem is false. Consider a critical counterexample, a pair of non-packable graphs $(G_1,G_2)$, with $G_2$ edge-minimal, satisfying the constraints of the theorem. We distinguish three cases, for each of which we derive an upper bound on $n$, given by one of the inequalities~\eqref{case1bound},~\eqref{case2bound} and~\eqref{case3bound}. At least one of these three inequalities should hold, so together they contradict the condition that $\max\left(\eqref{case1bound},\eqref{case2bound},\eqref{case3bound}\right) = \max(\eqref{eqn:blue1},\eqref{eqn:blue2},\eqref{eqn:blue3},\eqref{eqn:blue4})  < n$, thus proving the theorem.

\begin{enumerate}
\item \label{case:blue1} 
There exists a vertex $u \in [n]$ and there are labellings of $G_1$ and $G_2$ such that $u$ is incident to the unique purple edge and $|A^*(u)| \ge \alpha t \cdot  \Delta_2(\Delta_2+1)$.
\item \label{case:blue2}Case~\ref{case:blue1} does not hold and furthermore $|N_2(u)\cap N_2(v)| < (1-\epsilon) \cdot \Delta_2$ for some edge $uv \in E(G_2)$. 
\item \label{case:blue3} Neither of Cases~\ref{case:blue1} and~\ref{case:blue2} hold.
\end{enumerate} 
We now proceed with deriving upper bounds on $n$ for each of these three cases.

\paragraph{Bound for Case~\ref{case:blue1}.}
Choose a vertex $u\in [n]$ and labellings of $G_1$ and $G_2$ such that $u$ is incident to the unique purple edge and $|A^*(u)| \ge \alpha t \cdot \Delta_2(\Delta_2+1)$.  See Figure~\ref{fig:blue1} for a depiction of the argumentation in this case. From now on, we write $k:=|A^*(u)|/\Delta_2  -t (\Delta_2+1)$. Our first tool is Claim~\ref{clm:kclaim}, which yields that all $b\in B(u)$ satisfy $|N_1(b)\cap A^*(u)| \ge k$. Note that $k \ge 1$, since $\alpha> 1$. Our second tool is Corr\'adi's lemma, or rather Corollary~\ref{CorCor}, which we apply with $X=A^*(u)$ and $N= |B(u)|$ and with size $\ge k$ subsets $A_1,\ldots, A_N \subset X$ given by $N_1(b)\cap A^*(u)$, for all $b\in B(u)$. Note that $|A_i\cap A_j| \le t-1$ for all $i\neq j$, or else there would be a blue copy of $K_{2,t}$.

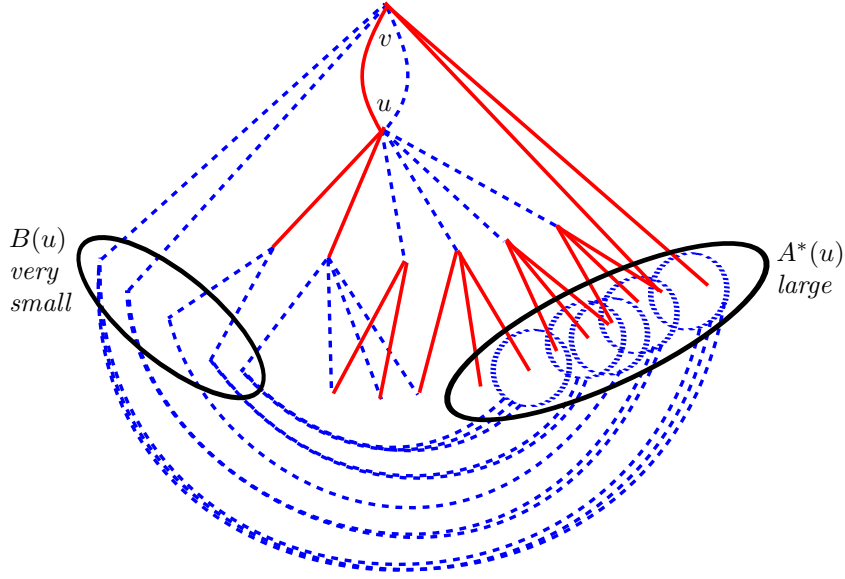
\begin{figure}
 \begin{center}
   \begin{overpic}[width=0.6\textwidth]{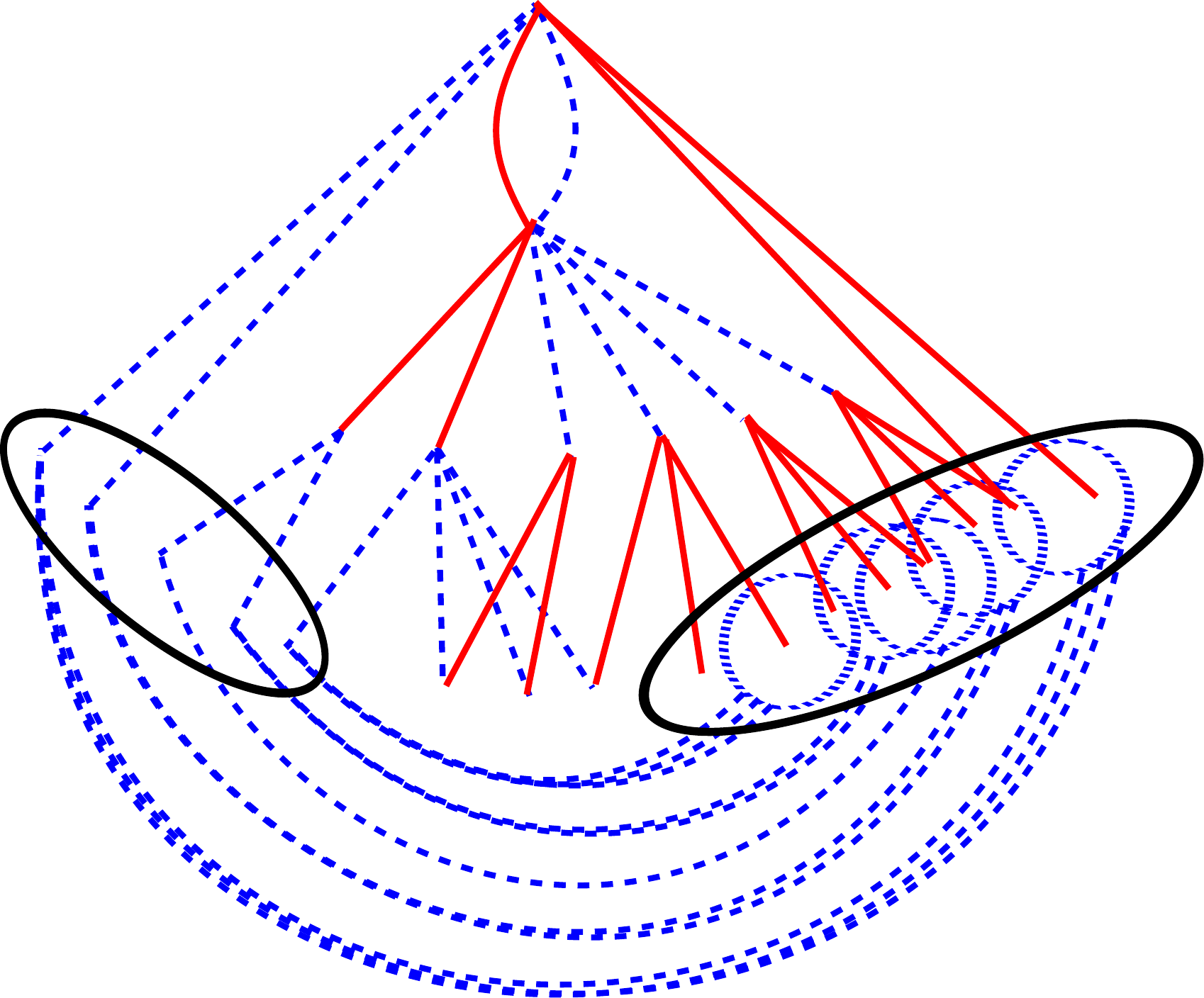}
    \put (42.3,76.5){ $v$}
    \put (42,67){ $u$}
    \put (-11,43) { \parbox{.5\textwidth}{$B(u)$ \\ \textit{very} \\ \textit{small}}  }
    \put (101,43) {\parbox{.5\textwidth}{$A^*(u)$ \\ \textit{large}} }
 \end{overpic}
  \caption{A depiction of Case~\ref{case:blue1} of Theorem~\ref{thm:blue}, that $|A^*(u)| = \Omega(\Delta_2^2)$ implies $|B(u)| = O(\Delta_2)$.\label{fig:blue1}}
\end{center}
\end{figure}

In order to apply Corollary~\ref{CorCor}, we need to check that its condition $k^2 > (t-1) \cdot  |A^*(u)|$ holds. For that, we 
write $\beta := |A^*(u)|/(t \Delta_2(\Delta_2+1))$, so that $k= (\beta-1)t(\Delta_2+1)$. Now
\begin{align*}
k^2 - (t-1) \cdot  |A^*(u)|
&= \left( (\beta-1) t (\Delta_2+1)  \right)^2 - \beta t \Delta_2(\Delta_2+1)(t-1) \\
&= \left( (\beta-1)^2- \gamma \cdot \beta \right) \cdot (t(\Delta_2+1))^2,
\end{align*} 
which is positive if and only if $(\beta-1)^2- \gamma \beta >0$, which holds true because $\beta \ge \alpha >\alpha^*$.
Thus, by Corollary~\ref{CorCor}, we obtain
\begin{align*}
|B(u)| \le |A^*(u)| \cdot \frac{k-(t-1)}{k^2 - (t-1) \cdot |A^*(u)|} =\frac{1-\frac{t-1}{k}}{\frac{k}{|A^*(u)|}-\frac{t-1}{k}}.
\end{align*}
The numerator and denominator of the right hand side are both positive, so we can bound and rearrange as follows:
\begin{align}\label{kopjethee2}
|B(u)| &\le  \left(\frac{k}{|A^*(u)|} - \frac{t-1}{k}\right)^{-1}
= \left( \frac{(\beta-1)t(\Delta_2+1)}{\beta t \Delta_2 (\Delta_2+1)}  - \frac{t-1}{(\beta-1)t(\Delta_2+1)} \right)^{-1}  \nonumber   \\
&= \Delta_2 \cdot \left( \frac{\beta-1}{\beta}  - \frac{1}{\beta-1} \cdot \frac{\Delta_2}{\Delta_2+1} \cdot \frac{t-1}{t} \right)^{-1}
 = \Delta_2 \cdot \left( \frac{\beta-1}{\beta}  - \frac{\gamma}{\beta-1}  \right)^{-1}    \nonumber  \\
&\le \Delta_2 \cdot \frac{\alpha (\alpha-1)}{(\alpha-1)^2- \gamma \alpha},
\end{align}
where the last step holds because $\beta \ge \alpha > \alpha^*$ and $\alpha^*$ is the larger singular point of $\frac{\beta (\beta-1)}{(\beta-1)^2-\gamma \beta}$, which is a decreasing function of $\beta$ for all $\beta> \alpha^*$.

Evaluating~\eqref{kopjethee2} and Claim~\ref{clm:truncatedblueneighbourhood} in the upper bound of Claim~\ref{clm:UpperBoundsN}\ref{clm:UpperBoundsN2} yields
 \begin{align}\label{case1bound} 
n &\le  |N_1^*(u)| + |N_2(u)| + |B(u)| + |N_2(N_1(u))| \nonumber \\
&\le  \left( t-1\right)  \cdot \Delta_2 + \Delta_2 + \frac{\alpha (\alpha-1)}{(\alpha-1)^2-\alpha} \cdot  \Delta_2  +  \Delta_1 \Delta_2 \nonumber \\
&= \left(t+\frac{\alpha (\alpha-1)}{(\alpha-1)^2-\alpha} \right) \cdot \Delta_2 + \Delta_1 \Delta_2.\end{align}

\paragraph{Bound for Case~\ref{case:blue2}.}
Choose labellings of $G_1$ and $G_2$ such that there is a unique purple edge $uv$ that satisfies $|N_2(u) \cap N_2(v)| < (1-\epsilon) \cdot \Delta_2$. Note that the inequalities $|A^*(u)| < \alpha t \cdot \Delta_2(\Delta_2+1)$ and $|A^*(v)| < \alpha t \cdot \Delta_2(\Delta_2+1)$ are satisfied as well, as a direct consequence of the assumptions of Case~\ref{case:blue2}.

We proceed with deriving a technical estimate on an intersection of neighbourhood sets. For each $x \in N_2(u) \setminus N_2(v)$ and $y \in N_2(v) \setminus N_2(u)$ we have $x\neq y$ and therefore absence of blue copies of $K_{2,t}$ implies the inequality $|N_1(x) \cap N_1(y)| \le t-1$. So 
\begin{align*}
|N_1( N_2(u)\setminus N_2(v)) \cap N_1(N_2(v)\setminus N_2(u))|
&\le \sum_{x \in N_2(u) \setminus N_2(v)} \sum_{y \in N_2(v) \setminus N_2(u)} |N_1(x) \cap N_1(y)| \\
&\le |N_2(u) \setminus N_2(v)| \cdot |N_2(v) \setminus N_2(u)| \cdot (t-1)\\
&\le (\Delta_2-|N_2(u) \cap N_2(v)|)^2 \cdot (t-1).
\end{align*}
Furthermore, since $|N_2(u) \cap N_2(v)| < (1-\epsilon) \cdot \Delta_2$,
\begin{align*}
|N_1(N_2(u)) \cap N_1(N_2(v)) | & \le |N_1(N_2 (u)\cap N_2(v))|  + |N_1( N_2(u)\setminus N_2(v)) \cap N_1(N_2(v)\setminus N_2(u))|  \\
&< \Delta_1 \cdot |N_2(u)\cap N_2(v)|  + (\Delta_2-|N_2(u)\cap N_2(v)|)^2 \cdot (t-1) \\
&\le \max_{p \in \left\{0,1,2,\ldots, \lfloor(1-\epsilon)\cdot \Delta_2 \rfloor \right\}} \left( \Delta_1 \cdot p  + (\Delta_2-p)^2 \cdot (t-1)  \right).
\end{align*}
See Figure~\ref{fig:blue2}.
Finally, we evaluate this in Claim~\ref{clm:UpperBoundsN}\ref{clm:UpperBoundsN3} to find the following bound on $n$:
\begin{align}\label{blablabla2}
 n &\le |A^*(v)| + |A^*(u)| +  |\left(N_2(u) \cup N_1(N_2(u)) \right) \cap \left( N_2(v) \cup N_1(N_2(v)) \right) | \nonumber \\
 &\le |A^*(v)| + |A^*(u)| +  |N_2(u)| + |N_2(v)| +  |N_1(N_2(u)) \cap N_1(N_2(v)) | \nonumber \\
&\le 2 \alpha t \cdot \Delta_2(\Delta_2+1) + 2\Delta_2 +  \max_{p \in \left\{0,1,2,\ldots, \lfloor(1-\epsilon)\cdot \Delta_2 \rfloor \right\}} \left( \Delta_1 \cdot p  + (\Delta_2-p)^2 \cdot (t-1)  \right).
\end{align}
In particular, this implies the slightly rougher bound 
\begin{align}\label{case2bound}
n \le 2 \alpha t \cdot \Delta_2(\Delta_2+1) + 2\Delta_2 +  (1-\epsilon) \cdot \Delta_1 \Delta_2  + \Delta_2^2 \cdot (t-1).
\end{align}

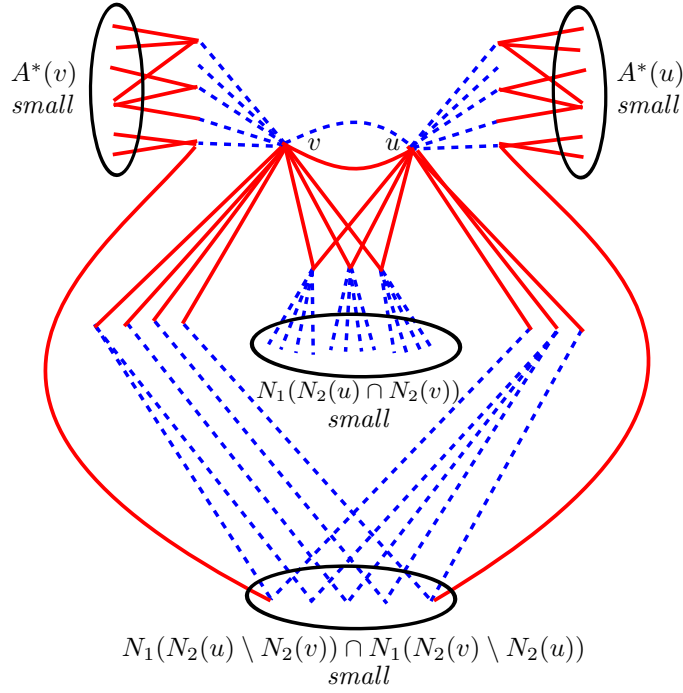
\begin{figure}
 \begin{center}
   \begin{overpic}[width=0.6\textwidth]{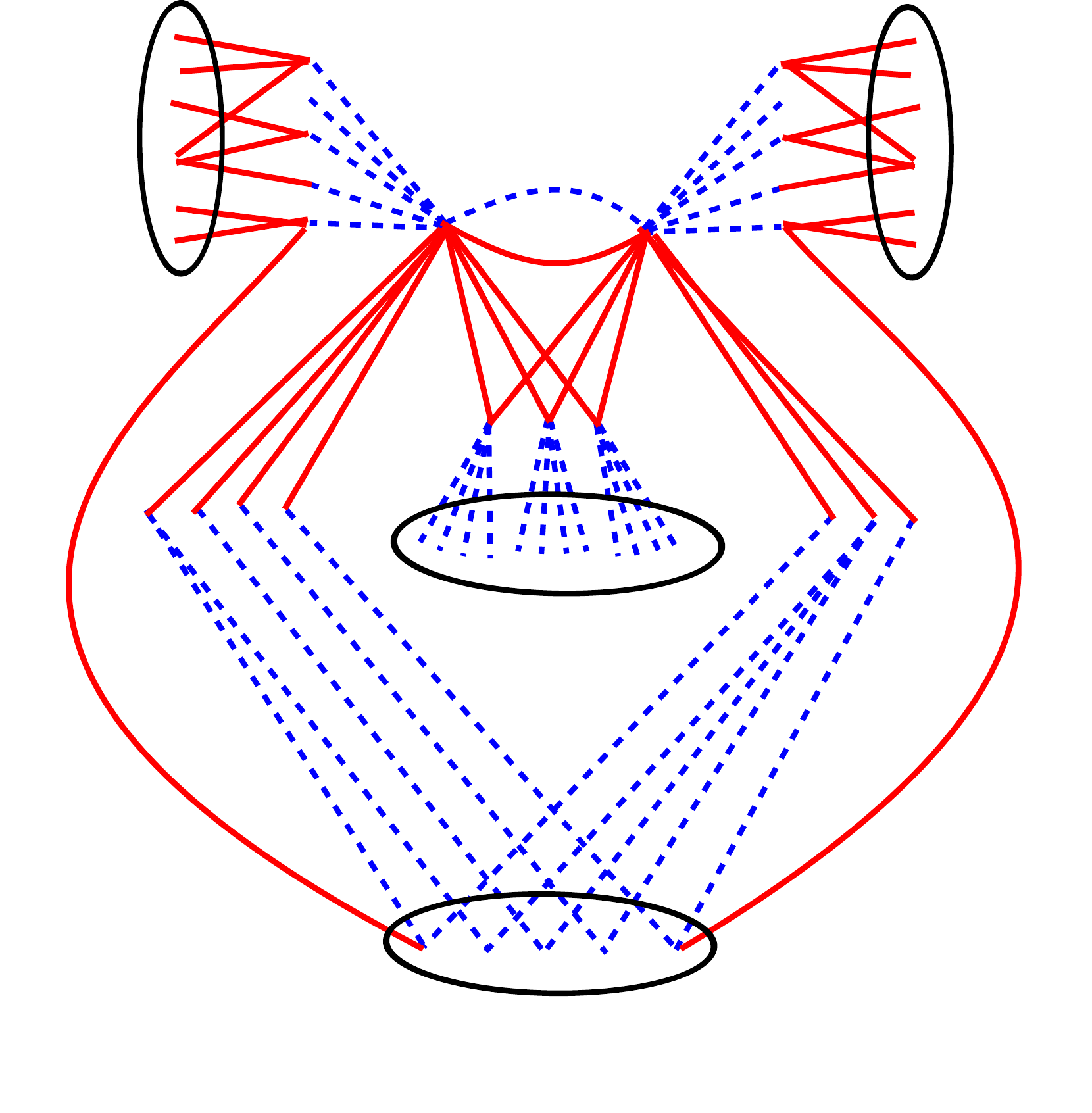}
    \put (42,79){ $v$}
    \put (53,79){ $u$}
    \put (0,87) { \parbox{.5\textwidth}{$A^*(v)$ \\ \textit{small}}  }
    \put (87,87) {\parbox{.5\textwidth}{$A^*(u)$ \\ \textit{small}} }
    \put (36,43.6){{\small $N_1(N_2(u) \cap N_2(v))$}}
    \put (46,39.5){\textit{small}}
    \put (17,6.8){$N_1( N_2(u)\setminus N_2(v)) \cap N_1(N_2(v)\setminus N_2(u))$}
    \put (46,2.8){\textit{small}}
 \end{overpic}
 \caption{A depiction of Case~\ref{case:blue2} of Theorem~\ref{thm:blue}, that $|N_1(N_2(u)) \cap N_1(N_2(v))|$ is small. \label{fig:blue2}}
\end{center}
\end{figure}

\paragraph{Bound for Case~\ref{case:blue3}.}
Choose a pair of labellings of $G_1$ and $G_2$ that induces a unique purple edge $uv$. The assumptions of this case imply, in particular, that in the red graph the neighbourhoods of each pair of adjacent vertices overlap significantly: $|N_2(x) \cap N_2(y)| \ge (1-\epsilon) \cdot \Delta_2$ for each $xy\in E(G_2)$. 

We will derive two consequences, namely the implication
\begin{align}\label{firstconsequence} \left(  |A^*(u)| \ge 1 + \Delta_2 + \frac{\epsilon \cdot \Delta_2 }{1-2 \epsilon} \right)  \implies  \left( |B(u)| \le (t-1) \cdot \Delta_2^2 \right) \end{align}
and the inequality
\begin{align}\label{secondconsequence}
|N_2(N_1(u))| \le \frac{1+\epsilon}{2} \Delta_1 \Delta_2 + \frac{1-\epsilon}{2} (t-1) \cdot \Delta_2^2  + \frac{3}{2} \Delta_2. 
\end{align}

We start with proving the statement~\eqref{firstconsequence}, the first consequence. 
See Figure~\ref{fig:blue3first}.
Suppose $a \in A^*(u) \setminus  N_2(u)$ has a red neighbour $x \in N_2(u)$. Then $ux$ and $ax$ are edges of $G_2$, so $|N_2(a) \cap N_2(x)| \ge (1-\epsilon) \Delta_2$ and $|N_2(u)\cap N_2(x)| \ge (1-\epsilon) \Delta_2$. Combining this with the obvious fact that $|N_2(x)| \le \Delta_2$ yields that 
\begin{align}\label{plasticzakje}
|N_2(a) \cap N_2(u)| \ge (1-2\epsilon)\cdot \Delta_2.
\end{align}
Let us define
\begin{align*}
A^{**}(u):= \left\{ a \in A^*(u) \setminus N_2(u) \,\mid\, a \text{ has a red neighbour in } N_2(u) \right\}.
\end{align*}
It follows from~\eqref{plasticzakje} that $\sum_{a \in A^{**}(u)} |N_2(a) \cap N_2(u)| \ge |A^{**}(u)| \cdot (1-2\epsilon) \cdot \Delta_2$, so
\begin{align*}
\sum_{x \in N_2(u)} |N_2(x)|
&\ge \sum_{x \in N_2(u)} |N_2(x) \cap N_2(u)| + \sum_{a \in A^{**}(u)} |N_2(a) \cap N_2(u)| \\
&\ge (1-\epsilon)\Delta_2 \cdot |N_2(u)| + |A^{**}(u)| \cdot (1-2\epsilon) \cdot \Delta_2,
\end{align*}
and (crucially) since $\sum_{x \in N_2(u)} | N_2(x)|  \le \Delta_2 \cdot |N_2(u)| $, it follows that
\begin{align}\label{astersteru}
|A^{**}(u)|     \le    \frac{|N_2(u)| \cdot \Delta_2 - (1-\epsilon)\cdot \Delta_2 |N_2(u)|}{(1-2\epsilon) \cdot \Delta_2}   
= \frac{\epsilon \cdot |N_2(u)| }{1- 2 \epsilon}.
\end{align}

\begin{figure}
 \begin{center}
   \begin{overpic}[width=0.6\textwidth]{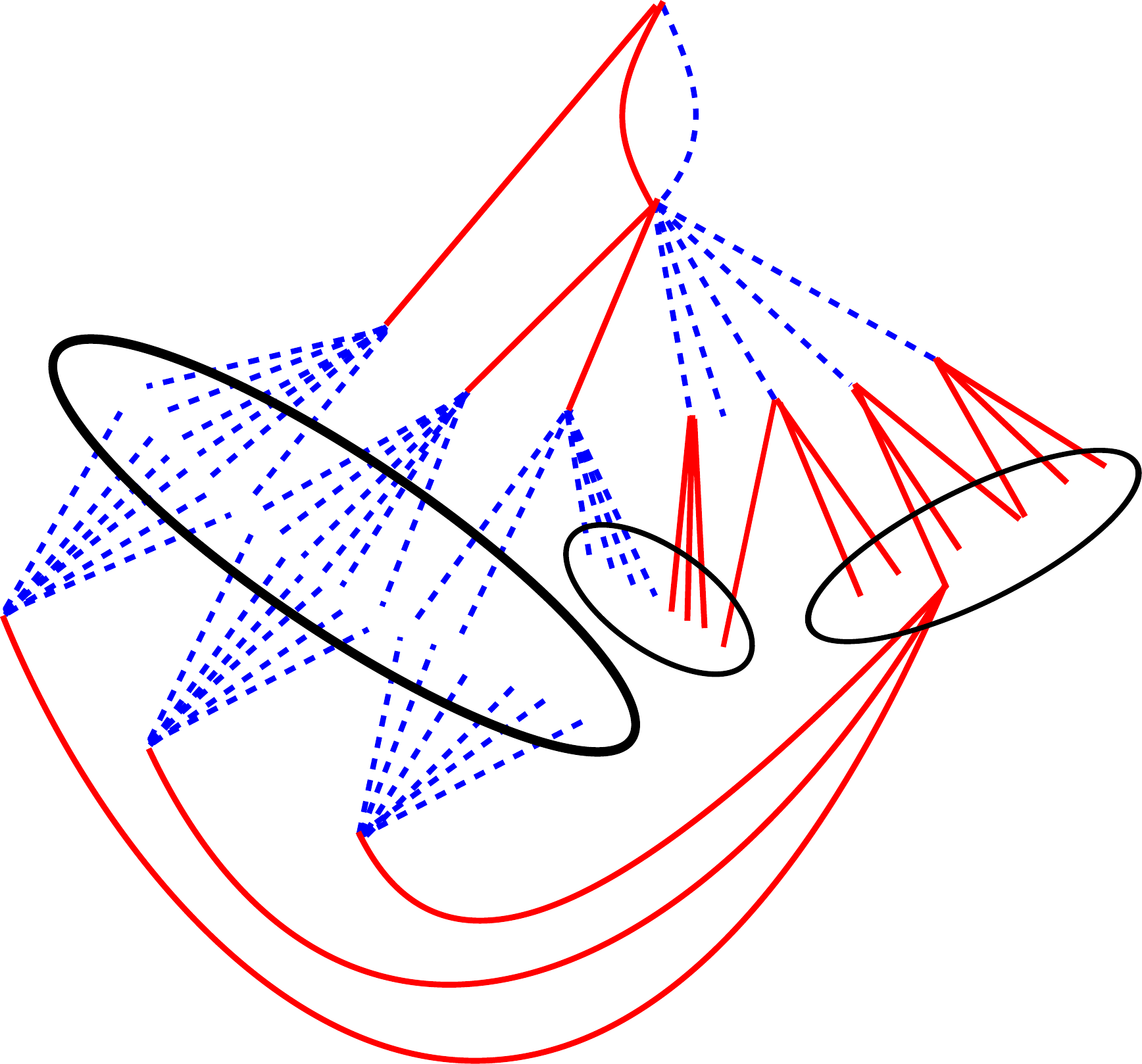}
    \put (55.5,87.5){ $v$}
    \put (55.3,78){ $u$}
    \put (84,41){$a$}
    \put (-8,60) { \parbox{.5\textwidth}{$B(u)$ \\
     \textit{small}}  }
    \put (96,41){ \parbox{.5\textwidth}{$A^*(u)$ \\ \textit{not very small}} }  
 \end{overpic}
\caption{A depiction of~\eqref{firstconsequence} in Case~\ref{case:blue3} of Theorem~\ref{thm:blue}.\label{fig:blue3first}}
\end{center}
\end{figure}
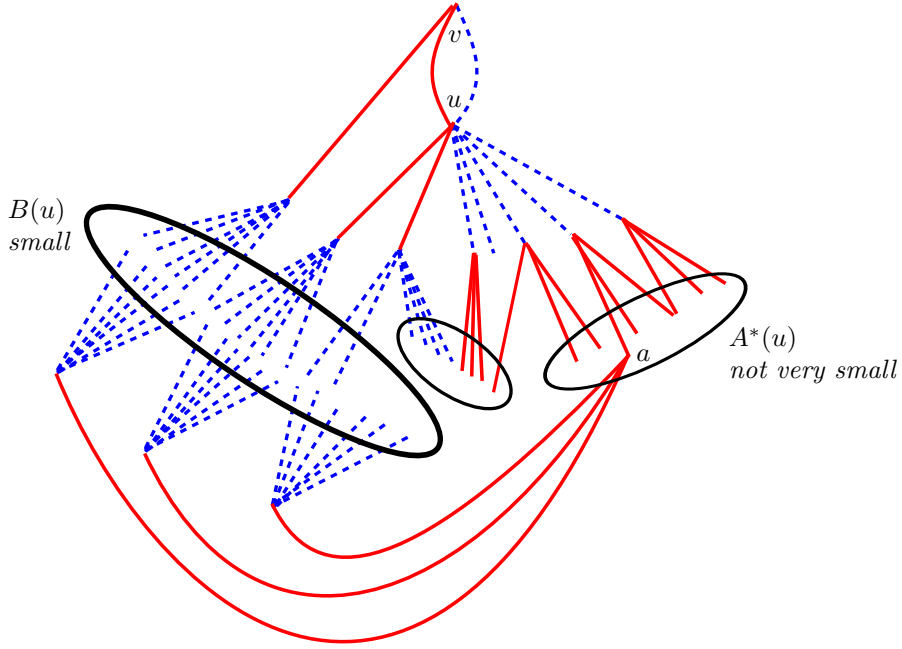

Next, suppose we would have that $|A^*(u)| \ge 1+ |N_2(u)| + |A^{**}(u)|$. Then there exists a vertex $a \in A^*(u) \setminus \left(N_2(u) \cup A^{**}(u)\right)$. By the definition of $A^{**}(u)$, this vertex satisfies $N_2(a) \cap N_2(u)= \emptyset$. Furthermore, since $a\in A^*(u)$, we have that for all $b \in B(u)$ there is a red--blue-link from $a$ to $b$. In other words, $B(u)=N_1(N_2(a)) \cap B(u)$. This implies that
$|B(u)|=|N_1(N_2(a)) \cap B(u)| \le |N_1(N_2(a)) \cap N_1(N_2(u))| \le  (t-1) \cdot \Delta_2^2$, where the last inequality is a consequence of the facts that $N_2(a) \cap N_2(u)= \emptyset$ and $G_1$ does not contain a copy of $K_{2,t}$. In summary, we have shown the implication
\begin{align}\label{protoimplicatie}
|A^*(u)| \ge 1+ |N_2(u)| + |A^{**}(u)| \implies  |B(u)| \le (t-1) \cdot \Delta_2^2. \end{align}
Combining~\eqref{astersteru} and~\eqref{protoimplicatie} yields our first desired main consequence~\eqref{firstconsequence}.

We now prove inequality~\eqref{secondconsequence}, the second consequence.
See Figure~\ref{fig:blue3second}.
First, the absence of blue copies of  $K_{2,t}$ implies that for every $x \in N_2(u) $ we have $|N_1(x) \cap N_1(u)| \le t-1$. Therefore 
\begin{align*}
| N_1(u) \cap N_1(N_2(u))| \le |N_2(u)| \cdot \max_{x \in N_2(u)} ( |N_1(x) \cap N_1(u)|) \le \Delta_2 \cdot (t-1).
\end{align*}
In other words, for at most $\Delta_2 \cdot (t-1)$ vertices $y \in N_1(u) $ there is a red--blue-link from $u$ to $y$. Recalling that there is a link from $u$ to every vertex (possibly with the exception of $v$), it follows that there are at least $h:=|N_1(u)| - (t-1) \Delta_2 -1$ vertices $y\in N_1(u)$ for which there is a blue--red-link (and no red--blue-link) from $u$ to $y$. In other words, $m:=|N_1(u) \cap N_2(N_1(u))| \ge h$. It follows from the definition of blue--red-link that any $y_1 \in N_1(u) \cap N_2(N_1(u))$ is connected to at least one other vertex $y_2 \in  N_1(u) \cap N_2(N_1(u))$  by a \textit{red edge}. If $m$ is even, this means that there exists a matching of $N_1(u)\cap N_2(N_1(u))$ consisting of red edges $y_1y_2, \ldots, y_{m-1}y_m$. Each of these edges has a large common red neighbourhood: for all $i\in \left\{ 1,3,5,\ldots, m-1\right\}$ it holds that $|N_2(y_1) \cup N_2(y_2)| = |N_2(y_1)| +| N_2(y_2)| - |N_2(y_1) \cap N_2(y_2)| \le \Delta_2 + \Delta_2 - (1-\epsilon) \Delta_2 = (1+\epsilon) \Delta_2$. So
\begin{align*}
|N_2 \left(  N_1(u) \cap N_2(N_1(u))   \right)| &= \left| \bigcup_{y \in N_1(u) \cap N_2(N_1(u))} N_2(y) \right| \\&\le \bigcup_{i \in \left\{  1,3, \ldots, m-1 \right\}} |N_2(y_i) \cup N_2(y_{i+1})| \le \frac{m}{2} \cdot \left(  1+ \epsilon \right) \cdot \Delta_2.
\end{align*}
If, on the other hand, $m$ is odd, then the same (or actually an even better) bound holds, because there exists a near-matching of $N_1(u)\cap N_2(N_1(u))$ with red edges $y_1y_2,\ldots, y_{m-4}y_{m-3}$ and a red $2$-path consisting of edges $y_{m-2}y_m$ and $y_{m-1}y_m$ satisfying $|N_2(y_{m-2}) \cup N_2(y_{m-1}) \cup N_2(y_{m})|  \le  |N_2(y_m)| + |N_2(y_{m-1}) \setminus N_2(y_m)| + |N_2(y_{m-2}) \setminus N_2(y_m)| \le (1+2 \epsilon) \Delta_2 \le \frac{3}{2} \cdot (1+\epsilon) \cdot \Delta_2$.

\begin{figure}
 \begin{center}
   \begin{overpic}[width=0.65\textwidth]{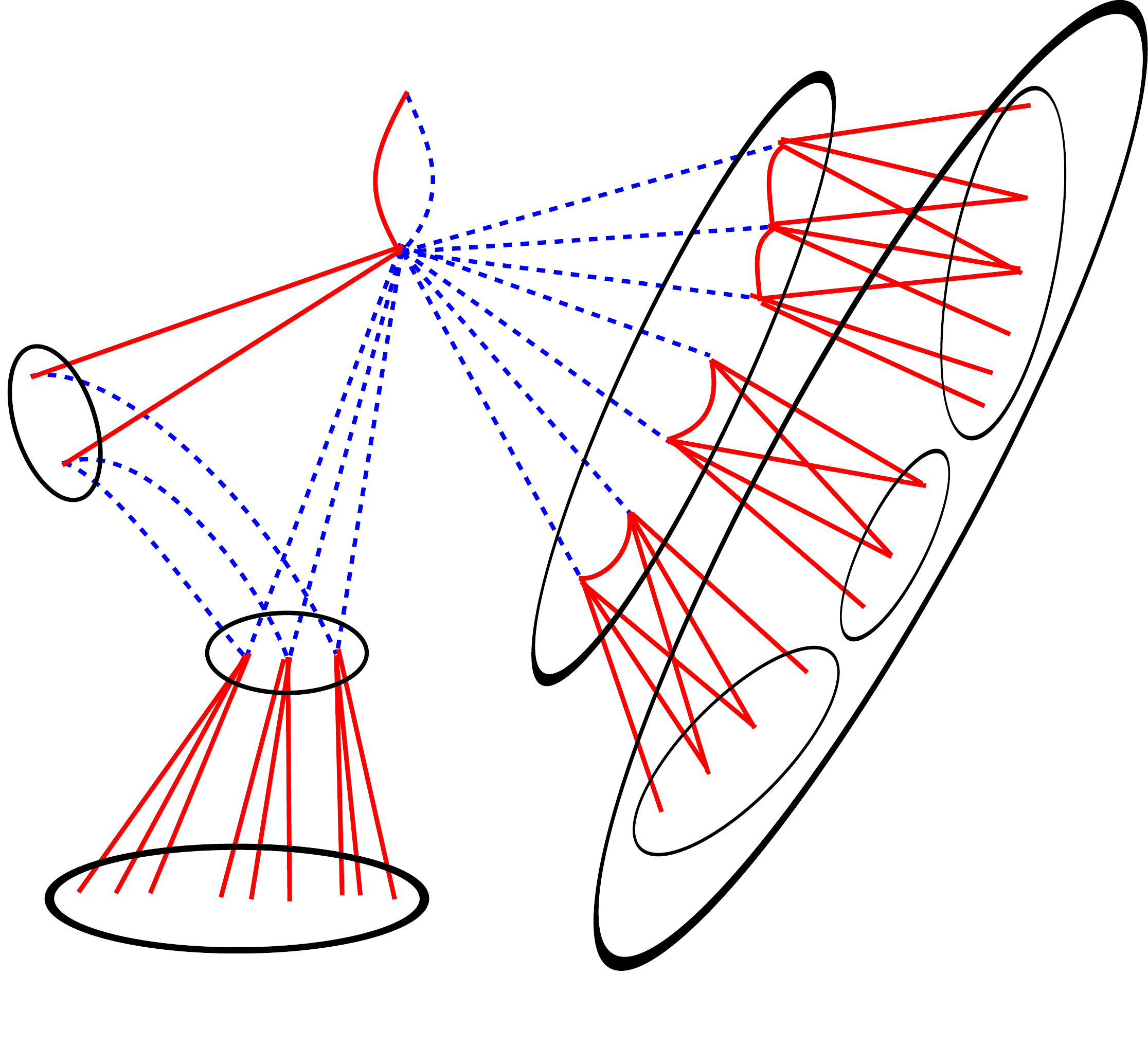}
    \put (33,78.5){ $v$}
    \put (33,71.5){ $u$}
    \put (-31.5,33) { \parbox{.5\textwidth}{{\small $N_1(u) \cap \left( N_1(N_2(u)) \backslash N_2(N_1(u)) \right)$ }\\
     \textit{very small}}  }
    \put (54,87){ $N_1(u) \cap N_2(N_1(u)))$  } 
    \put (81,33){\parbox{.5\textwidth}{{\small $N_2(N_1(u) \cap N_2(N_1(u)))$ }\\ \textit{small}}  } 
    \put (-7,2.5){\parbox{.5\textwidth}{{\small $N_2\left( N_1(u) \cap (N_1(N_2(u)) \backslash N_2(N_1(u))) \right)$ } \\ \textit{small}}  } 
    \put (59.5,22) {{\tiny small}}
    \put (76.5,46) {{\tiny small}}
    \put (86,70.5) {{\tiny small}}
 \end{overpic}
\caption{A depiction of~\eqref{secondconsequence} in Case~\ref{case:blue3} of Theorem~\ref{thm:blue}. \label{fig:blue3second}}
\end{center}
\end{figure}
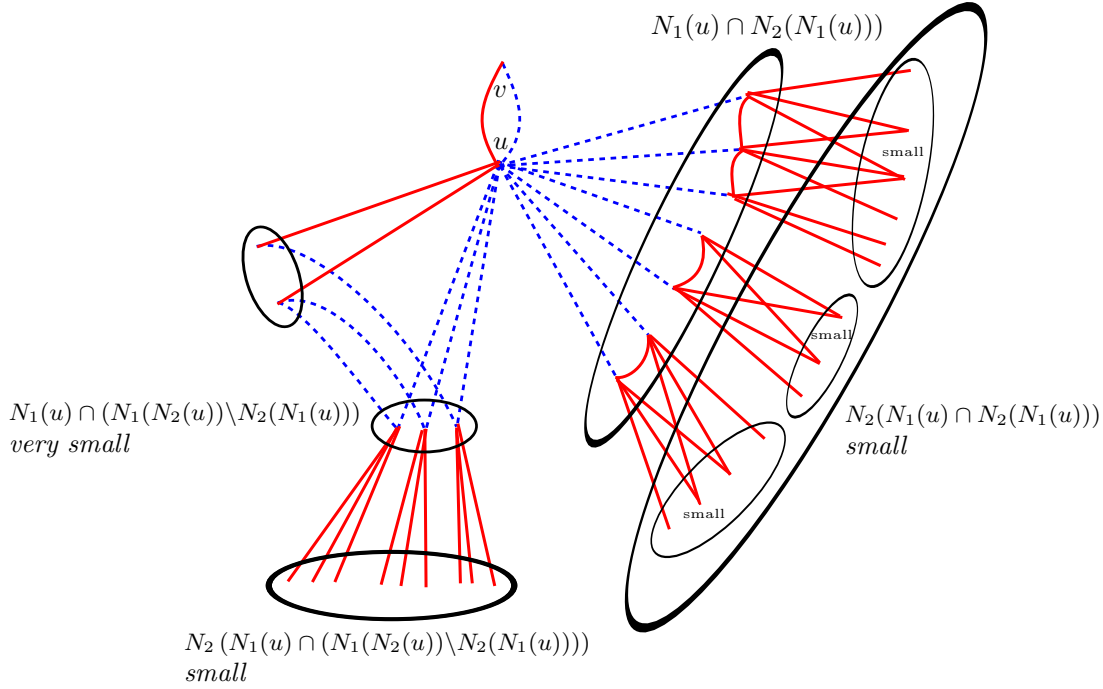

Last, note that 
\begin{align*}
|N_1(u) \cap (N_1(N_2(u)) \setminus N_2(N_1(u)))| = |N_1(u)| - m - \mathbbm{1}_{\left\{  \nexists\text{ link from $u$ to $v$} \right\}} \le |N_1(u)| - m.
\end{align*}

We are now ready to derive~\eqref{secondconsequence}:
\begin{align*} 
|N_2(N_1(u))| 
&\le  |N_2 \left(  N_1(u) \cap N_2(N_1(u)) \right)| +   |N_2 \left(  N_1(u) \cap (N_1(N_2(u)) \setminus N_2(N_1(u)))   \right)| + |N_2(v)| \\
&\le  \frac{m}{2} \cdot (1+ \epsilon) \cdot \Delta_2  + \left(|N_1(u)| - m \right) \cdot \Delta_2  + \Delta_2 \quad =:   g(m).
\end{align*}

Since $\Delta_2 \ge 0$ and $\epsilon <1/2$, the function $g(x)$ is nonincreasing on the whole of $\mathbb{R}$. Since $h \le m$, it follows that $g(m) \le g(h)$. So 
\begin{align*}
|N_2(N_1(u))|
&\le g(|N_1(u)| - (t-1) \Delta_2 -1) \\
&= \frac{1+\epsilon}{2} \cdot (|N_1(u)| - (t-1) \cdot \Delta_2 -1) \cdot \Delta_2+ (t-1)\cdot \Delta_2^ 2 + 2 \Delta_2 \\
&\le \frac{1+\epsilon}{2} \cdot \Delta_1 \Delta_2 + \frac{1-\epsilon}{2} \cdot (t-1) \cdot \Delta_2^2 + \frac{3-\epsilon}{2} \cdot \Delta_2,
\end{align*}
as desired. 

Finally, we evaluate~\eqref{firstconsequence} and~\eqref{secondconsequence} in the bounds on $n$ given by Claim~\ref{clm:UpperBoundsN}, parts~\ref{clm:UpperBoundsN1} and~\ref{clm:UpperBoundsN2}, to obtain
\begin{align}\label{case3bound}
n &\le \min \left( |N_1(N_2(u))| + |A^*(u)| + |N_2(u)|, \ |N_2(N_1(u))|  + |N_2(u)| + |N_1^*(u)| + |B(u)| \right) \nonumber  \\
& \le  \min \left(\Delta_1 \Delta_2 + \Delta_2 + |A^*(u)|, \ \frac{1+ \epsilon}{2} \Delta_1 \Delta_2 + \frac{1- \epsilon}{2} (t-1) \Delta_2^2 + \left(t+ \frac{3-\epsilon}{2}\right) \cdot \Delta_2+ |B(u)|  \right) \nonumber  \\
&= \Delta_1 \Delta_2 + \Delta_2 + \min \left( |A^*(u)|, \ |B(u)| + \left(t+\frac{1-\epsilon}{2}\right)\cdot \Delta_2   -  \frac{1- \epsilon}{2}\left( \Delta_1 \Delta_2 -  (t-1) \Delta_2^2 \right)    \right) \nonumber \\
&= \Delta_1 \Delta_2 + \Delta_2 + \max \left(1+\Delta_2 + \frac{\epsilon \Delta_2}{1-2\epsilon}, \ \frac{3-\epsilon}{2} (t-1)\cdot  \Delta_2^2  -  \frac{1- \epsilon}{2} \Delta_1 \Delta_2  + \left(t+\frac{1-\epsilon}{2}\right)\cdot \Delta_2  \right),
\end{align}
where we employed~\eqref{firstconsequence} and~\eqref{secondconsequence} only in the last line.
\qed


\subsection{Proof of Theorem~\ref{thm:bluered}}

Suppose the theorem is false. Consider a critical counterexample, a pair of non-packable graphs $(G_1,G_2)$ satisfying the constraints of the theorem, such that there is a near-packing with a unique purple edge $uv$. We distinguish two cases, Cases~\ref{case:bluered1} and~\ref{case:bluered2}. From the first we derive the inequality~\eqref{case1bound_SB} and from the second we obtain the inequality~\eqref{case2bound_SB}. Together they contradict the condition that $\max(\eqref{eqn:bluered1},\eqref{eqn:bluered2}) < n$, thus proving the theorem.

\begin{enumerate}
\item \label{case:bluered1}
$|A^*(u)| \ge \alpha t \cdot \Delta_2(\Delta_2+1)$ or $|A^*(v)| \ge \alpha t \cdot \Delta_2(\Delta_2+1)$.
\end{enumerate}
Without loss of generality, we assume $|A^*(u)| \ge \alpha t \cdot \Delta_2(\Delta_2+1)$. From here the proof is the same as for Case~\ref{case:blue1} in the proof of Theorem~\ref{thm:blue}, leading to the same bound,
 \begin{align}\label{case1bound_SB} 
n \le \left(t+\frac{\alpha (\alpha-1)}{(\alpha-1)^2-\alpha} \right) \cdot \Delta_2 + \Delta_1 \Delta_2.\end{align}

\begin{enumerate}
\setcounter{enumi}{1}
\item \label{case:bluered2}
Case~\ref{case:bluered1} does not hold.
\end{enumerate}
From here we proceed almost exactly as for Case~\ref{case:blue2} in the proof of Theorem~\ref{thm:blue}, the difference being that instead of the upper bound $|N_2(u)\cap N_2(v)| < (1-\epsilon) \cdot \Delta_2$ we use $|N_2(u)\cap N_2(v)| < s$, which holds due to the additional condition $\Delta_2^\vartriangle < s$. (Compare with~\eqref{case2bound}.) It follows that 
\begin{align}\label{case2bound_SB}
n \le 2 \alpha t \cdot \Delta_2(\Delta_2+1) + 2\Delta_2 +  \Delta_1 \cdot (s-1) + \Delta_2^2 \cdot (t-1). \qedhere
\end{align}
\qed

\subsection{Concluding remarks}\label{sec:conclusion}

We wish to make the following remarks about Theorems~\ref{thm:blue} and~\ref{thm:bluered}.

\begin{itemize}
\item
In Theorem~\ref{thm:blue}, the bottleneck is the quantity~\eqref{eqn:blue2}, which corresponds to the bound~\eqref{case2bound} of Case~\ref{case:blue2}. So improving in this case would improve the overall bound on $n$, albeit not by much.
\item
The condition in Theorem~\ref{thm:bluered} that $\Delta_2^{\vartriangle}<s$ is equivalent to ``$|N_2(x) \cap N_2(y)| <s$ for all $xy \in E(G_2)$''. With a little adaptation, we can replace this by the weaker but perhaps obscure condition that $G_2$ has \textit{no} subgraph $G_2^!$ such that $|N_2(x) \cap N_2(y)| \ge s$ for \textit{all} $xy \in E(G_{2}^!)$. Indeed, this property is invariant under edge removal, and so holds for an edge-minimal critical counterexample, which therefore has an edge $uv$ with $|N(u)\cap N(v)| <s$, for which we can choose labellings such that $uv$ is the unique purple edge. From here, one again proceeds exactly as in Case~\ref{case:blue2} of the proof of Theorem~\ref{thm:blue}.
\item
Theorem~\ref{thm:bluered} yields a better bound than Theorem~\ref{thm:blue} only if $\Delta_1$ is much larger than $\Delta_2$ and $s$, $t$ are both small.
\item
By taking $G_2$ to be a collection of (nearly) equal-sized cliques, 
Corollary~\ref{cor:blue0} implies that, if $G$ is a $K_{2,t}$-free graph of maximum degree $\Delta$ with $\Delta \geq \sqrt{17 t} \cdot \sqrt{n}$, then the equitable chromatic number of $G$ is at most $\Delta$. Note that this result cannot be obtained by the result of Hajnal and Szmer\'edi on equitable colourings~\cite{HaSz70}.
\end{itemize}

The BEC conjecture notwithstanding, naturally one might wonder whether Theorem~\ref{thm:blue}, or rather Corollary~\ref{cor:blue}, could be improved according to a weaker form of the BEC condition, as was the case for $d$-degenerate $G_1$~\cite{BKN08}. In other words, it would be interesting to improve upon the $\Omega(\Delta_1\Delta_2)$ terms appearing in each of~\eqref{eqn:blue1}--\eqref{eqn:blue4}. We leave this to further study, but point out the following constructions where $G_1$ has low maximum codegree, which mark boundaries for this problem.
\begin{itemize}
\item
When $n$ is even, there are non-packable pairs $(G_1,G_2)$ of graphs where $G_1$ is a perfect matching (so $\Delta^\wedge_1 = 0$) and $2 \Delta_1 \Delta_2 = n$, cf.~\cite{KaKo07}.
\item
Bollob\'as, Kostochka and Nakprasit~\cite{BKN05} exhibited a family of non-packable pairs $(G_1, G_2)$ of graphs where $G_1$ is a forest (so $\Delta^\wedge_1 = 1$) and $\Delta_1 \ln \Delta_2 \ge cn$ for some $c>0$. 
\item
If $\Delta^\wedge(G) = 1$, then the chromatic number of $G$ satisfies $\chi(G) = O(\Delta(G)/\ln \Delta(G))$ as $\Delta(G)\to\infty$, and there are standard examples having arbitrarily large girth that show this bound to be sharp up to a constant factor, cf.~\cite[Ex.~12.7]{MoRe02}. Since the equitable chromatic number is at least the chromatic number, these examples moreover yield non-packable pairs $(G_1, G_2)$ of graphs having $\frac{\Delta_1}{\ln \Delta_1} (\Delta_2+1) \ge cn$ for some $c > 0$ and $\Delta^\wedge_1 = 1$.
\end{itemize}

Since the examples can also have the maximum adjacent codegree $\Delta^\vartriangle_1$ being zero, this last remark hints at another natural line to pursue, which could significantly extend both the result of Csaba~\cite{Csa07} and a result of Johansson~\cite{Joh96}. If $\Delta_1$ is large enough and $G_1$ is triangle-free, is some condition of the form $\frac{\Delta_1}{\ln \Delta_1} (\Delta_2+1) = cn$ for some constant $c>0$ sufficient for $G_1$ and $G_2$ to pack?

\bibliographystyle{abbrv}
\bibliography{packing}

\begin{thebibliography}{10}

\bibitem{AiBr93}
M.~Aigner and S.~Brandt.
\newblock Embedding arbitrary graphs of maximum degree two.
\newblock {\em J. London Math. Soc. (2)}, 48(1):39--51, 1993.

\bibitem{BoEl78}
B.~Bollob{\'a}s and S.~E. Eldridge.
\newblock Packings of graphs and applications to computational complexity.
\newblock {\em J. Combin. Theory Ser. B}, 25(2):105--124, 1978.

\bibitem{BKN05}
B.~Bollob{\'a}s, A.~Kostochka, and K.~Nakprasit.
\newblock On two conjectures on packing of graphs.
\newblock {\em Combin. Probab. Comput.}, 14(5-6):723--736, 2005.

\bibitem{BKN08}
B.~Bollob{\'a}s, A.~Kostochka, and K.~Nakprasit.
\newblock Packing {$d$}-degenerate graphs.
\newblock {\em J. Combin. Theory Ser. B}, 98(1):85--94, 2008.

\bibitem{Cat74}
P.~A. Catlin.
\newblock Subgraphs of graphs. {I}.
\newblock {\em Discrete Math.}, 10:225--233, 1974.

\bibitem{Cat76}
P.~A. Catlin.
\newblock {\em Embedding subgraphs and coloring graphs under extremal degree
  conditions}.
\newblock ProQuest LLC, Ann Arbor, MI, 1976.
\newblock Thesis (Ph.D.)--The Ohio State University.

\bibitem{Cor69}
K.~Corr\'{a}di.
\newblock Problem at {S}chweitzer competition.
\newblock {\em Mat. Lapok}, 20:159--162, 1969.

\bibitem{Csa07}
B.~Csaba.
\newblock On the {B}ollob\'as-{E}ldridge conjecture for bipartite graphs.
\newblock {\em Combin. Probab. Comput.}, 16(5):661--691, 2007.

\bibitem{CSS03}
B.~Csaba, A.~Shokoufandeh, and E.~Szemer{\'e}di.
\newblock Proof of a conjecture of {B}ollob\'as and {E}ldridge for graphs of
  maximum degree three.
\newblock {\em Combinatorica}, 23(1):35--72, 2003.
\newblock Paul Erd{\H{o}}s and his mathematics (Budapest, 1999).

\bibitem{HaSz70}
A.~Hajnal and E.~Szemer{\'e}di.
\newblock Proof of a conjecture of {P}. {E}rd{\H o}s.
\newblock In {\em Combinatorial theory and its applications, {II} ({P}roc.
  {C}olloq., {B}alatonf\"ured, 1969)}, pages 601--623. North-Holland,
  Amsterdam, 1970.

\bibitem{Joh96}
A.~Johansson.
\newblock Asymptotic choice number for triangle-free graphs.
\newblock Technical Report 91-5, DIMACS, 1996.

\bibitem{KaKo07}
H.~Kaul and A.~Kostochka.
\newblock Extremal graphs for a graph packing theorem of {S}auer and {S}pencer.
\newblock {\em Combin. Probab. Comput.}, 16(3):409--416, 2007.

\bibitem{KKY08}
H.~Kaul, A.~Kostochka, and G.~Yu.
\newblock On a graph packing conjecture by {B}ollob\'as, {E}ldridge and
  {C}atlin.
\newblock {\em Combinatorica}, 28(4):469--485, 2008.

\bibitem{MoRe02}
M.~Molloy and B.~Reed.
\newblock {\em Graph Colouring and the Probabilistic Method}, volume~23 of {\em
  Algorithms and Combinatorics}.
\newblock Springer-Verlag, Berlin, 2002.

\bibitem{SaSp78}
N.~Sauer and J.~Spencer.
\newblock Edge disjoint placement of graphs.
\newblock {\em J. Combin. Theory Ser. B}, 25(3):295--302, 1978.

\end{thebibliography}

\end{document}